\documentclass[sn-mathphys,iicol]{sn-jnl_modified}

\usepackage{graphicx}%
\usepackage{multirow}%
\usepackage{amsmath,amssymb,amsfonts}%
\usepackage{amsthm}%
\usepackage{mathrsfs}%
\usepackage[title]{appendix}%
\usepackage{xcolor}%
\usepackage{textcomp}%
\usepackage{manyfoot}%
\usepackage{booktabs}%
\usepackage{algorithm}%
\usepackage{algorithmicx}%
\usepackage{algpseudocode}%
\usepackage{listings}%
\usepackage{physics}
\usepackage{enumerate}
\usepackage{dsfont}

\theoremstyle{definition}
\newtheorem{definition}{Definition}[section]
\newtheorem{proposition}[definition]{Proposition}
\newtheorem{lemma}[definition]{Lemma}
\newtheorem{theorem}[definition]{Theorem}
\newtheorem{corollary}[definition]{Corollary}
\newtheorem{remark}[definition]{Remark}
\newtheorem{example}[definition]{Example}
\newtheorem{alg}[definition]{Algorithm}

\newcommand{\N}{\mathbb{N}}
\newcommand{\R}{\mathbb{R}}
\newcommand{\sph}{\mathbb{S}}
\newcommand{\B}{\mathcal{B}}

\newcommand{\ind}{\mathds{1}}
\newcommand{\E}{\mathbb{E}}
\newcommand{\gap}{\textsf{gap}}
\newcommand{\IAT}{\text{IAT}}
\newcommand{\G}{\mathcal{G}}
\renewcommand{\H}{\mathcal{H}}
\newcommand{\M}{\mathcal{M}}
\renewcommand{\L}{\mathcal{L}}
\renewcommand{\d}{\text{d}}
\newcommand{\dif}[2]{\frac{\d #1}{\d #2}}
\newcommand{\supp}{\,\textsf{supp}}

\newcommand{\ra}{\rightarrow}

\newcommand{\LRA}{\Leftrightarrow}

\newcommand\ccint[2]{[#1,#2]}
\newcommand\ooint[2]{(#1,#2)}
\newcommand\ocint[2]{(#1,#2]}
\newcommand\coint[2]{[#1,#2)}
\newcommand\ccintv[2]{\left[#1,#2\right]}
\newcommand\oointv[2]{\left(#1,#2\right)}

\begin{document}

\title[Dimension-independent spectral gap of polar slice sampling]{Dimension-independent spectral gap of polar slice sampling}

\author*[1]{\fnm{Daniel} \sur{Rudolf}}\email{daniel.rudolf@uni-passau.de}

\author[2]{\fnm{Philip} \sur{Sch\"ar}}\email{philip.schaer@uni-jena.de}

\affil*[1]{\orgdiv{Faculty of Computer Science and Mathematics}, \orgname{University of Passau}, \orgaddress{\street{Innstrasse 33}, \city{Passau}, \postcode{94032}, \country{Germany}}}

\affil[2]{\orgdiv{Microscopic Image Analysis Group}, \orgname{Friedrich Schiller University Jena}, \orgaddress{\street{Kollegiengasse 10}, \city{Jena}, \postcode{07743}, \country{Germany}}}

\abstract{
Polar slice sampling, a Markov chain construction for approximate sampling, performs, under suitable assumptions on the target and initial distribution, provably independent of the state space dimension. We extend the aforementioned result of \cite{PolarSS} by developing a theory which identifies conditions, in terms of a generalized level set function, that imply an explicit lower bound on the spectral gap even in a general slice sampling context. Verifying the identified conditions for polar slice sampling yields a lower bound of 1/2 on the spectral gap for arbitrary dimension if the target density is rotationally invariant, log-concave along rays emanating from the origin and sufficiently smooth. The general theoretical result is potentially applicable beyond the polar slice sampling framework.
}

\keywords{MCMC, slice sampling, spectral gap}

\pacs[MSC Classification]{65C05, 60J05, 60J22}

\maketitle

\section{Introduction} \label{Sec:Intro}

We consider the problem of approximate sampling of a distribution, which is, in the context of Bayesian inference, a permanently present challenge. The goal is to simulate realizations of a random variable that is distributed according to a probability measure of interest $\pi$ defined on $(\R^d, \B(\R^d))$, with $d \in \N$ and $\B(\R^d)$ being the Borel $\sigma$-algebra of $\R^d$. We assume to be able to evaluate a not necessarily normalized Lebesgue density of $\pi$ given by $\varrho: \R^d \ra \R_+$, i.e., for any $A \in \B(\R^d)$ we have
\begin{equation}
	\pi(A) = \frac{1}{C} \int_{A} \varrho(x) \d x ,
	\label{Eq:pi}
\end{equation}
where
\begin{equation*}
	C:= \int_{\R^d} \varrho(x) \d x
	\in \ooint{0}{\infty}
\end{equation*}
is an unknown normalization constant. Because of the only partial knowledge about $\varrho$, the standard approach for dealing with such sampling problems is to construct a Markov chain with limit distribution $\pi$.

The slice sampling methodology (see e.g.~\cite{BesagGreen}) provides a framework for the construction of a Markov chain $(X_n)_{n\in\N_0}$ with $\pi$-reversible transition kernel, where the distribution of $X_n$ converges (under weak regularity conditions) to the distribution of interest, see e.g. \cite{SliceConv}. We focus here on polar slice sampling (PSS) that exploits the almost surely well-defined factorization $\varrho(x) = p_0(x) p_1(x)$ with
\begin{equation}
	p_0(x) := \norm{x}^{1-d} , \qquad
	p_1(x) := \norm{x}^{d-1} \varrho(x),
	\label{Eq:PSS_fac_intro}
\end{equation}
where $\norm{\cdot}$ denotes the Euclidean norm in $\R^d$.
The choice of this particular factorization in the slice sampling context has been proposed in \cite{PolarSS}. The resulting transition mechanism of the corresponding Markov chain $(X_n)_{n\in\N_0}$ on $(\R^d, \B(\R^d))$ can be presented as follows.\\[-1ex]

\begin{alg} \label{alg}
	Given the target density $\varrho  = p_0 \, p_1$ and the current state $X_{n-1}=x$, PSS w.r.t. $\varrho$ generates the next instance $X_n$ by the following two steps:
	\begin{enumerate}
		\item Draw an auxiliary random variable $T_{n}$ with respect to (w.r.t.) the uniform distribution on $\ooint{0}{p_1(x)}$. Call the realization $t_n$ and define the super level set
		\begin{equation*}
			L(t_n,p_1): = \{z \in \R^d \mid p_1(z) > t_n \}.
		\end{equation*}
		\item Draw $X_{n}$ from the distribution $\mu_{t_n}$ on $\R^d$ that is given by
		\begin{equation*}
			\mu_{t_n}(A) := \frac{\int_{A \cap L(t_n,p_1)} p_0(z) \, \d z}{\int_{L(t_n,p_1)} p_0(z) \, \d z}.
		\end{equation*}
	\end{enumerate}
\end{alg}

\cite{PolarSS} offer an implementation of Algorithm~\ref{alg} using polar coordinates and an acceptance rejection approach w.r.t.~radius and spherical element. Admittedly, already in easy examples, the acceptance probability can be very small, which turns the implementation to be computationally demanding, especially in the case of large $d$. However, in \cite{GPSS} a Gibbsian polar slice sampling methodology has been proposed that on the one hand mimics PSS and on the other hand offers a computationally feasible scheme. Actually our investigation is very much driven by the hope to carry the result about the dimension-independence of PSS over to this related approach.
To illustrate the empirically dimension-independent performance of PSS we present the following numerical illustration.\\[-1ex]

\textbf{Motivating numerical illustration.}
We consider the polar and uniform slice sampling Markov chains. The transition mechanism of the latter is exactly as stated in Algorithm~\ref{alg}, except that it sets the factorization of $\varrho$ to $p_0(x) := 1$ and $p_1(x) := \varrho(x)$ for any $x\in \R^d$ (in contrast to \eqref{Eq:PSS_fac_intro}).
For both the unimodal target density\footnote{Numerical experiments in that setting were already conducted in \cite[Section 3]{PolarSS}.} $\varrho(x) = \exp(-\norm{x})$ and the volcano-shaped target density $\varrho(x) = \exp(-(\norm{x}-2)^2)$, we plot in Figure~\ref{Fig:motivation} proxies of the \emph{integrated autocorrelation time} (IAT) of the aforementioned Markov chains, depending on the state space dimension $d$. Since the IAT characterizes the asymptotic mean squared error\footnote{For a function $g\colon \R^d\to \R$ as well as a generic Markov chain $(Y_n)_{n\in\N}$ with initial distribution $\pi$ and $\pi$-reversible transition kernel $P$, the integrated autocorrelation time $\IAT_{g,P}$ satisfies, for example by virtue of \cite[Proposition~3.26]{RudolfDiss}, that
	\begin{equation*}
		\lim_{n\to \infty} n \cdot \E\abs{ \frac{1}{n} \sum_{i=1}^n g(Y_i) - \pi(g) }^2 = \IAT_{g,P} \cdot (\pi(g^2) - \pi(g)^2) ,
	\end{equation*}
	where $\pi(g)=\int_{\R^d} g \d \pi$ with $\pi(g^2)<\infty$ and $\IAT_{g,P}$ defined in \eqref{Eq:IAT} below.
} (and the asymptotic variance within CLTs) of the Markov chain Monte Carlo time average w.r.t.~summary function~$g\colon \R^d\to \R$ we can conclude that the smaller it is, the `better' is the Markov chain. We consider $g(x) = \norm{x}$.
\begin{figure}
	\begin{center}
		\includegraphics[width=0.49\textwidth]{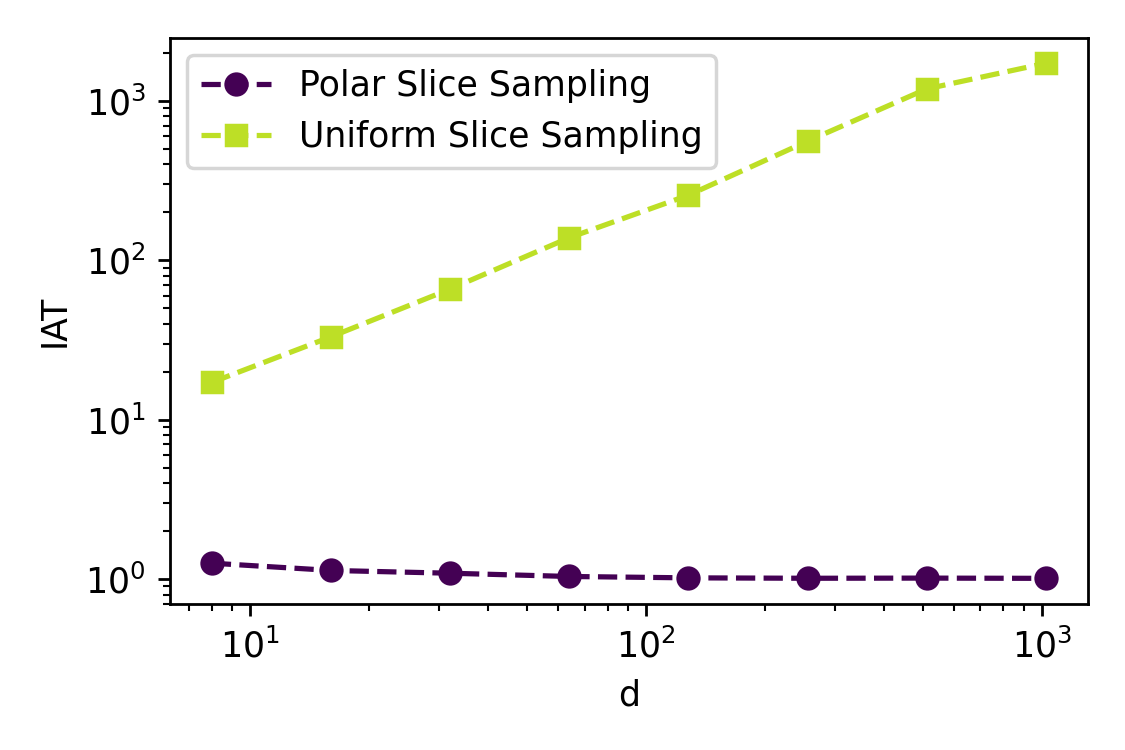}
		\includegraphics[width=0.49\textwidth]{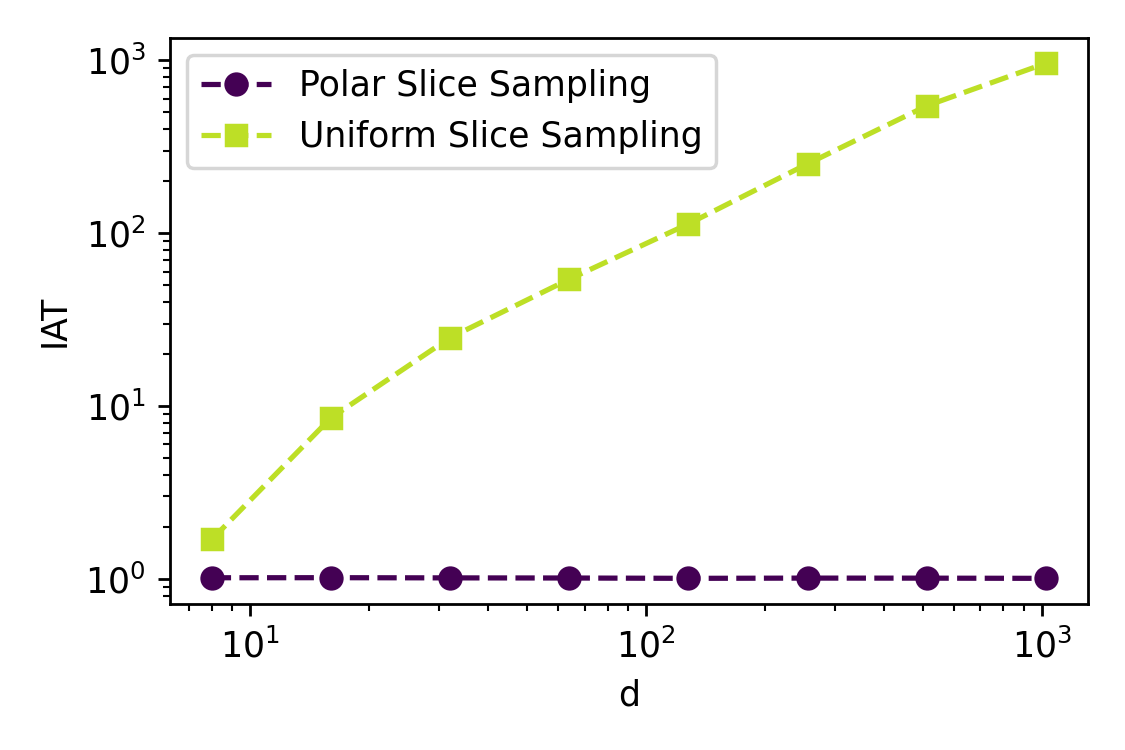}
		\caption{Sample space dimension $d$ versus approximations of the integrated autocorrelation time $\text{IAT}_{g,P}$, as defined in \eqref{Eq:IAT}, computed using the heuristic described in \cite[Chapter 11.5]{Gelman}. The figure on top depicts IATs for the target density $\varrho(x) = \exp(-\norm{x})$ and that below IATs for target density $\varrho(x) = \exp(-(\norm{x} - 2)^2)$. Both figures use the summary function $g(x) = \norm{x}$. Each plotted point represents an average over $n_{\text{rep}} = 10$ separate runs of the samplers, using $n_{\text{it}} = 10^5$ iterations for each sampler and repetition.}
		\label{Fig:motivation}
	\end{center}
\end{figure}
In Figure~\ref{Fig:motivation} it is clearly visible that the IAT of PSS is constantly slightly larger than $1$ regardless of the dimension. In contrast to that the IAT of uniform slice sampling (USS) increases as the state space dimension increases, showing that the efficiency of the corresponding Markov chain degenerates with increasing dimension. That is also theoretically confirmed in \cite{SlavaUSS}. However, it is particularly surprising that PSS exhibits such remarkably `good' constant dimension behavior.\\[-1.5ex]

\cite{PolarSS} explain this behavior in their Theorem~7 with Remark~8 by proving that for any rotational invariant $\varrho$ that is log-concave along rays emanating from the origin and any initial state $x\in \R^d$ satisfying $p_1(x)\geq 0.01 \cdot \left(\sup_{w\in\R^d} p_1(w)\right)$ one has
\begin{equation}
	\norm{P^{525}(x,\cdot) - \pi}_{\text{tv}} \leq 0.01 ,
	\label{Eq:tv_small}
\end{equation}
where $P^{525}(x,\cdot) = \mathbb{P}(X_{525}\in \cdot\mid X_0=x)$ and
\begin{equation*}
	\norm{P^{525}(x,\cdot) - \pi}_{\text{tv}} := \sup_{A \in \B(\R^d)} \abs{P^{525}(x,A) - \pi(A)}
\end{equation*}
is the total variation distance between $\pi$ and $P^{525}(x,\cdot)$.
Actually, in \cite[Theorem~7]{PolarSS}, there is no rotational invariance assumption, but an asymmetry parameter appears, and, as long as this does not depend on the dimension, the former result holds by changing the $525$ to some larger number still independent of $d$.

We refine and extend the result \eqref{Eq:tv_small} by providing in the same setting a lower bound of $1/2$ of the spectral gap of the Markov operator of PSS. Even though we postpone the definition and discussion about the spectral gap of a Markov chain (or corresponding transition kernel) to Section~\ref{Sec:BG}, we want to briefly motivate here that it is a crucial object.
A quantitative lower bound of the gap of a transition kernel $P$ corresponding to a Markov chain $(X_n)_{n \in \N_0}$ with stationary distribution $\pi$ implies a number of useful properties. These include geometric convergence (with explicit convergence rate) of the distribution of $X_n$ to $\pi$ as $n \ra \infty$ (see \eqref{Eq:geom_converg} or \cite[Theorem~2.1]{RoRo_hybrid_MCs} or \cite[Theorem~1]{gallegos2022equivalences}), a non-asymptotic error bound for the classical Markov chain Monte Carlo time average \cite[Theorem 3.41]{RudolfDiss}, a central limit theorem (CLT) \cite{gapCLT} and an estimate of the CLT asymptotic variance \cite{gapCLTasvar}. Moreover, it implies an explicit upper bound of the IAT (which follows for example by \eqref{Eq:asymp_var} below) and therefore explains the motivating numerical illustration straightforwardly.

Our investigation builds upon the work of \cite{SlavaUSS}. There, among other things, a duality technique that gives sufficient conditions for quantitative lower bounds on the spectral gap of USS has been developed. We extend the duality argumentation to general slice sampling (with a non-specified factorization $\varrho = \varrho_0 \,\varrho_1$) and apply the resulting theory to PSS. More precisely, in the general setting we offer a sufficient condition of the spectral gap in terms of properties of the function $\ell_{\varrho_0,\varrho_1} \colon \ooint{0}{\infty} \to \R_+$ given by
\begin{equation*}
	\ell_{\varrho_0,\varrho_1}(t)
	:= \int_{L(t,\varrho_1)} \varrho_0(x) \,\d x , \qquad t \in \ooint{0}{\infty} ,
\end{equation*}
see Theorem~\ref{Thm:gap_estimate} and Definition~\ref{Def:Lambda_k} below. Applying this result in the context of the PSS factorization yields the dimension-independent lower bound of $1/2$ of the spectral gap, as long as $\varrho$ is rotational invariant, log-concave along rays emanating from the origin and sufficiently smooth, see Theorem~\ref{Thm:PSS_gap_est}.

We now provide some guidance trough the structure of the paper. In the next section we introduce our notation and define all required Markov chain related objects. Afterwards, in Section \ref{Sec:gap_est_tool}, we discuss how a number of theoretical results from \cite{SlavaUSS} translate from USS to the general case. In Section \ref{Sec:PSS_gap_est}, we apply the results from Section \ref{Sec:gap_est_tool} to PSS, thereby proving a lower bound on its spectral gap. Concluding remarks with a discussion of our results and an outlook can be found in Section~\ref{Sec:Con}.

\section{Preliminaries} \label{Sec:BG}

We introduce our notation and state some useful facts. All appearing random variables map from a joint sufficiently rich probability space onto their respective state space. With $\lambda$ we denote the Lebesgue measure on $(\R,\B(\R))$ and for the surface measure on the Euclidean unit sphere $\sph^{d-1}$ equipped with its natural Borel $\sigma$-algebra $\B(\sph^{d-1})$ we write $\sigma_{d-1}$. We provide details about kernels.

Let $(G,\G)$ and $(H,\H)$ be measurable spaces. A \textit{transition kernel} on $G \times \H$ is a mapping $P: G \times \H \ra [0,1]$ such that $P(\cdot,A)$ is a measurable function for all $A \in \H$ and $P(x,\cdot) \in \M_1(H)$ for all $x \in G$, where $\M_1(H)$ denotes the set of probability measures on $(H,\H)$. Let $P$ be a transition kernel on $G \times \G$, then $P$ acts on measurable functions $g: G \ra \R$ by
\begin{equation} \label{Eq:MO_right}
	P g (x)
	:= \int_G g(y) P(x, \d y) , 
	\quad x \in G .
\end{equation}
Let $Q$ be a transition kernel on $G \times \H$ and let $\xi \in \M_1(G)$, then $Q$ acts on $\xi$ as 
\begin{equation*}
	\xi Q(A)
	:= \int_G Q(x,A) \xi(\d x) ,
	\quad A \in \H,
\end{equation*}
and defines a probability measure, i.e., $\xi Q \in \M_1(H)$.
Moreover, the \textit{tensor product} of $\xi$ and $Q$ is defined as the probability measure on $(G \times H, \G \times \H)$ determined by
\begin{align*}
	(\xi \otimes Q)(A \times B)
	&:= \int_A Q(x,B) \xi(\d x) \\
	&= \int_A \int_B Q(x,\d y) \xi(\d x) ,
	\; A \in \G, B \in \H .
\end{align*}
Additionally, let $R$ be a transition kernel on $H \times \G$, then the \textit{composition} of $Q$ and $R$ is the transition kernel $Q R$ on $G \times \G$ defined by
\begin{equation*}
	Q R(x, A)
	:= \int_H R(y, A) Q(x, \d y) ,
	\quad x \in G, A \in \G .
\end{equation*}
Using this, for a transition kernel $P$ on $G\times \G$, one recursively defines $P^1 := P$ and $P^n := P P^{n-1}$ for $n \geq 2$.

For a Markov chain $(X_n)_{n\in\N_0}$ on $(G,\G)$ with transition kernel $P$ and initial distribution $\xi\in \M_1(G)$ it is well known that the probability measure $\xi P^n$ coincides with the distribution of $X_n$. We say that the transition kernel $P$ (and the corresponding Markov chain) has invariant distribution $\pi \in \M_1(G)$ if $\pi P = \pi$. Moreover, it is reversible w.r.t. $\pi$ if $(\pi\otimes P)(A\times B) = (\pi \otimes P)(B\times A)$ for all $A,B\in \G$.

We turn to the definition of the spectral gap of a transition kernel $P$ on $G \times \G$ that is reversible w.r.t.~$\pi \in \M_1(G)$ and therefore has $\pi$ as invariant distribution.
With $L_2(\pi)$ we denote the space of measurable functions $g\colon G \ra \R$ satisfying 
\begin{equation*}
	\norm{g}_{2,\pi}^2 :=  \int_G g(x)^2 \pi(\d x) < \infty .
\end{equation*}
Note that $\norm{\cdot}_{2,\pi}$ is a norm on the quotient space of $L_2(\pi)$ under the equivalence relation identifying functions that coincide $\pi$-a.e. It is induced by the inner product $\langle \cdot, \cdot \rangle_{\pi}$ on $L_2(\pi)$ defined by
\begin{equation*}
	\langle g, h \rangle_{\pi}
	:= \int_G g(x) h(x) \pi(\d x) .
\end{equation*}
Observe that $P$ acting on functions $g: G \ra \R$ via $g \mapsto P g$ as in \eqref{Eq:MO_right} defines a linear operator mapping from $L_2(\pi)$ into $L_2(\pi)$. Interpreting $\pi$ as a transition kernel that is constant in its first argument, $\pi$ also induces a linear operator mapping from $L_2(\pi)$ into $L_2(\pi)$, specifically by
\begin{equation}
	\pi g (x)
	= \int_{\R^d} g(y) \pi(\d y) .
	\label{Eq:pi_as_operator}
\end{equation}
This allows us to define the spectral gap of $P$ as
\begin{equation*}
	\gap_{\pi}(P)
	:= 1 - \norm{P - \pi}_{L_2(\pi) \ra L_2(\pi)} ,
\end{equation*}
where $\norm{\cdot}_{L_2(\pi) \ra L_2(\pi)}$ denotes the operator norm w.r.t.~$\norm{\cdot}_{2,\pi}$.

With these formal notions at hand, we may now explicitly state some of the consequences of spectral gap estimates for $\pi$-reversible Markov chains that we already mentioned in the introduction. For example, it is well known, see e.g.~\cite[Lemma~2]{Novak}, that it implies geometric convergence, i.e.,
\begin{equation}
	\norm{\xi P^n -\pi}_{\text{tv}} 
	\leq (1-\gap_{\pi}(P))^n \norm{ \dif{\xi}{\pi}-1 }_{2,\pi},
	\label{Eq:geom_converg}
\end{equation}
where $\norm{\cdot}_{\text{tv}}$ again denotes the total variation distance. 

An explicit lower bound of $\gap_\pi(P)$ also leads to a mean squared error bound of the Markov chain Monte Carlo sample average, see \cite[Theorem~3.41]{RudolfDiss}. Moreover, a classical result of \cite{gapCLT} states that if the initial distribution is the invariant distribution $\pi$ and $g\in L_2(\pi)$ then the $\sqrt{n}$-scaled sample average error 
\begin{equation*}
	\sqrt{n} \left( \frac{1}{n}\sum_{i=1}^{n}g(X_i) -\pi(g) \right)
\end{equation*}
converges weakly to the normal distribution $\mathcal{N}(0,\sigma_{g,P}^2)$ with mean zero and variance
\begin{equation*}
	\sigma_{g,P}^2 = \langle (I+P)(I-P)^{-1}(g-\pi(g)),(g-\pi(g))\rangle_{\pi} ,
\end{equation*}
where $I$ denotes the identity map. The significant quantity $\sigma_{g,P}^2$ satisfies
\begin{equation}
	\IAT_{g,P}  \cdot \norm{g-\pi(g)}_{2,\pi}^2 
	= \sigma_{g,P}^2 
	\leq \frac{2 \norm{g-\pi(g)}_{2,\pi}^2}{\gap_{\pi}(P)},
	\label{Eq:asymp_var}
\end{equation}
where 
\begin{equation}
	\IAT_{g,P} = 1 + 2\sum_{k\geq 1} \gamma_g(k) ,
	\label{Eq:IAT}
\end{equation}
with correlations
\begin{equation*}
	\gamma_g(k) = \text{Corr}(g(X_0),g(X_{k})) ,
\end{equation*}
denotes the integrated autocorrelation time.

\section{Spectral Gap Estimate}

In this section, we first introduce general slice sampling and derive a tool that can be used to establish spectral gap estimates. We then apply it to PSS.

\subsection{General Slice Sampling} \label{Sec:gap_est_tool}

For the probability measure of interest $\pi\in \M_1(\R^d)$ we assume to have an almost sure (w.r.t.~the Lebesgue measure) factorization of the not necessarily normalized density of the form
\begin{equation*}
	\varrho(x) = \varrho_0(x) \varrho_1(x) , 
	\qquad x \in \R^d,
\end{equation*}
with measurable functions $\varrho_i\colon \R^d \to \R_+$ for $i=0,1$. General slice sampling exploits this representation by (essentially) performing the two steps of Algorithm~\ref{alg}, except that $\varrho_0$ takes the role of $p_0$ and $\varrho_1$ the role of $p_1$. We refer to the 1st step as $T$-update and to the 2nd one as $X$-update.
The transition kernels $U_T$ on $\R^d \times \B(\ooint{0}{\infty})$ and $U_X$ on $\ooint{0}{\infty} \times \B(\R^d)$ that correspond to the aforementioned $T$- and $X$-update of Algorithm~\ref{alg} are given by 
\begin{align*}
	U_T(x,A)
	&
	= \frac{\lambda(A\cap \ooint{0}{\varrho_1(x)})}{\lambda(\ooint{0}{\varrho_1(x)})}
	= \frac{\int_A \ind_{L(t,\varrho_1)}(x)\, \d t}{\varrho_1(x)} , \\
	U_X(t,B)
	&= \frac{\int_{B\cap L(t,\varrho_1)} \varrho_0(x) \,\d x}{\int_{L(t,\varrho_1)} \varrho_0(x) \, \d x} =: \mu_t(B) ,
\end{align*}
where $x\in\R^d, A\in \B(\ooint{0}{\infty})$ and $t\in \ooint{0}{\infty},\; B\in \B(\R^d)$.
Thus, the Markov chain $(X_n)_{n\in\N_0}$ of the slice sampling for $\pi$ has transition kernel $P_X = U_T U_X$.
Moreover, the sequence of auxiliary random variables $(T_n)_{n \in \N}$, see the 2nd step of Algorithm~\ref{alg}, is (also) a Markov chain on $(\ooint{0}{\infty},\B(\ooint{0}{\infty}))$ with transition kernel $P_T=U_X U_T$.

We now elaborate on how the investigation of the spectral gap of USS by \cite{SlavaUSS} translates to general slice sampling. As a first step, we provide the invariant distribution of $P_T$, which follows by standard arguments that are also delivered for the convenience of the reader.\\[-1ex]

\begin{lemma} \label{Lem:aux_chain_inv_dist}
	Let $\widetilde{\pi}\in \M_1(\ooint{0}{\infty})$ be determined by the probability density function
	\begin{equation*}
		\widetilde{\varrho}(t)
		= C^{-1} \int_{L(t,\varrho_1)} \varrho_0(x) \d x , 
		\qquad t\in \ooint{0}{\infty},
	\end{equation*}
	such that $\widetilde{\pi}(\d t) = \widetilde{\varrho}(t)\, \d t$. Then $P_T$ is reversible w.r.t.~$\widetilde{\pi}$.
\end{lemma}
\begin{proof}
	Fix $A \in \B(\R^d)$ and $B \in \B(\ooint{0}{\infty})$. Note that
	\begin{align*}
		&(\pi \otimes U_T)(A \times B) \\
		&= \int_A U_T(x, B) \pi(\d x) \\  
		&= \int_A \frac{\int_B \ind_{L(t,\varrho_1)}(x)\, \d t}{\varrho_1(x)} \cdot C^{-1} \, \varrho_0(x) \varrho_1(x) \d x \\ 
		&= C^{-1} \int_B \int_{A\cap L(t,\varrho_1)} \varrho_0(x)\, \d x \d t \\
		&= \int_B U_X(t,A) \widetilde{\varrho}(t)\, \d t . 
	\end{align*}
	This yields
	\begin{align*}
		1 
		&= (\pi \otimes U_T)(\R^d \times \ooint{0}{\infty}) \\
		&= \int_{0}^\infty U_X(t,\R^d) \widetilde{\varrho}(t)\, \d t 
		= \int_{0}^\infty \widetilde{\varrho}(t)\, \d t ,
	\end{align*}
	proving that $\widetilde{\varrho}$ is indeed normalized. Plugging this fact into the former computation shows
	\begin{equation*}
		(\pi \otimes U_T)(A \times B)
		= \int_B U_X(t,A) \,\widetilde{\pi}(\d t)
		= (\widetilde{\pi} \otimes U_X)(B \times A).
	\end{equation*}
	For any measurable $F\colon \R^d \times \ooint{0}{\infty} \to \R$ (for which one of the following integrals exists) the latter equation extends to
	\begin{align*}
		&\int_{\R^d} \int_0^\infty F(x,t)\, U_T(x, \d t)\, \pi(\d x) \\
		&= \int_0^\infty	\int_{\R^d}  F(x,t)\, U_X(t, \d x)\, \widetilde{\pi}(\d t).
	\end{align*}
	Therefore, we obtain
	\begin{align*}
		&(\widetilde{\pi} \otimes P_T)(B_1 \times B_2) \\
		&= (\widetilde{\pi} \otimes U_X U_T)(B_1 \times B_2) \\
		&= \int_{0}^\infty \ind_{B_1}(t) U_X U_T (t, B_2) \widetilde{\pi}(\d t) \\
		&= \int_{0}^\infty \int_{\R^d} \ind_{B_1}(t) U_T(x,B_2) U_X(t, \d x) \widetilde{\pi}(\d t) \\
		&= \int_{\R^d} \int_{0}^\infty \ind_{B_1}(t) U_T(x,B_2) U_T(x, \d t) \pi(\d x) \\
		&= \int_{\R^d} U_T(x,B_1) U_T(x,B_2) \pi(\d x) 
	\end{align*}
	for $B_1, B_2 \in \B(\ooint{0}{\infty})$. The last expression is symmetric in $B_1$ and $B_2$, such that a backwards argumentation interchanging the roles of $B_1$ and $B_2$ shows that $P_T$ is reversible w.r.t.~$\widetilde{\pi}$. 
\end{proof}
Note that by the same steps one can prove the well-known fact that $P_X$ is reversible w.r.t.~$\pi$. Having this we are able to formulate our spectral gap duality result. The statement follows by the application of Lemmas~\ref{Lem:update_kernels_adjoint} and \ref{Lem:linop_identities} that can be found in the appendix.\\[-1ex]

\begin{theorem} \label{Thm:gap_identity_X_T}
	The linear operators $P_X \colon L_2(\pi) \to L_2(\pi)$ and $P_T \colon L_2(\widetilde{\pi})\to L_2(\widetilde{\pi})$ induced by the corresponding transition kernels via \eqref{Eq:MO_right} satisfy
	\begin{equation*}
		\gap_{\pi}(P_X) = \gap_{\widetilde{\pi}}(P_T) .
	\end{equation*}
\end{theorem}
\begin{proof}
	Define the linear operators $W := U_T - \widetilde{\pi}$ and $W^{\ast} := U_X - \pi$. By Lemmas \ref{Lem:update_kernels_adjoint} and \ref{Lem:linop_identities} (i), we know that $W^{\ast}$ is the adjoint operator of $W$. Furthermore, by Lemma \ref{Lem:linop_identities} (ii) and the fact that $P_X = U_T U_X$, we get
	\begin{align*}
		W W^{\ast}
		&= (U_T - \widetilde{\pi})(U_X - \pi) \\
		&= U_T U_X - U_T \pi - \widetilde{\pi} U_X + \widetilde{\pi} \pi \\
		&= U_T U_X - \pi
		= P_X - \pi .
	\end{align*}
	Analogously, by Lemma \ref{Lem:linop_identities} (iii) and the fact that $P_T = U_X U_T$, we get
	\begin{align*}
		W^{\ast} W
		&= (U_X - \pi)(U_T - \widetilde{\pi}) \\
		&= U_X U_T - U_X \widetilde{\pi} - \pi U_T + \pi \widetilde{\pi} \\
		&= U_X U_T - \widetilde{\pi}
		= P_T - \widetilde{\pi} .
	\end{align*}
	Now, denoting by $\norm{\cdot}_{L_2(\cdot) \ra L_2(\cdot)}$ the respective operator norms and applying some well-known facts from functional analysis \cite[Theorem V.5.2]{Werner}, we obtain
	\begin{align*}
		\norm{P_X - \pi}_{L_2(\pi) \ra L_2(\pi)}
		&= \norm{W W^{\ast}}_{L_2(\pi) \ra L_2(\pi)} \\
		&= \norm{W^{\ast}}_{L_2(\pi) \ra L_2(\widetilde{\pi})}^2 \\
		&= \norm{W}_{L_2(\widetilde{\pi}) \ra L_2(\pi)}^2 \\
		&= \norm{W^{\ast} W}_{L_2(\widetilde{\pi}) \ra L_2(\widetilde{\pi})} \\
		&= \norm{P_T - \widetilde{\pi}}_{L_2(\widetilde{\pi}) \ra L_2(\widetilde{\pi})} .
	\end{align*}
	By the spectral gap's definition, this implies the claimed identity.
\end{proof}

Keeping in mind that $\varrho=\varrho_0 \, \varrho_1$ we verify next that $P_T$ (essentially) only depends on the target distribution $\pi$ through a univariate function $\ell_{\varrho_0,\varrho_1}$. Here $\ell_{\varrho_0,\varrho_1}$ can be considered as an immediate extension of the level-set function from \cite[e.g.~Lemma~2.4]{SlavaUSS} into the general slice sampling setting. We start with a proper definition.\\[-1ex]

\begin{definition} \label{Def:ell}
	For a factorized density $\varrho = \varrho_0\, \varrho_1$ we define the generalized level-set function
	$\ell_{\varrho_0,\varrho_1}: \ooint{0}{\infty} \ra \R_+$ by
	\begin{equation*}
		\ell_{\varrho_0,\varrho_1}(t)
		:= \int_{L(t,\varrho_1)} \varrho_0(x) \d x, \qquad t \in \ooint{0}{\infty}.
	\end{equation*}
\end{definition}

Observe that $- \ell_{\varrho_0,\varrho_1}$ is a non-decreasing function on $\ooint{0}{\infty}$. Therefore it can serve as the integrator in a Lebesgue-Stieltjes integral, see e.g.~\cite[Section 1.3.2]{Athreya}. We now derive the aforementioned representation of $P_T$ that only depends on $\ell_{\varrho_0,\varrho_1}$.\\[-1ex]

\begin{theorem} \label{Thm:P_T_identity}
	For any $t > 0$ and $B \in \B(\ooint{0}{\infty})$ we have
	\begin{equation*}
		P_T(t,B)
		= \frac{1}{\ell_{\varrho_0,\varrho_1}(t)} \int_t^{\infty} \frac{\lambda(B \, \cap \ooint{0}{s})}{s}\; \d(-\ell_{\varrho_0,\varrho_1})(s) ,
	\end{equation*}
	where the right-hand side denotes a Lebesgue-Stieltjes integral w.r.t.~$- \ell_{\varrho_0,\varrho_1}$.
\end{theorem}
\begin{proof}
	Define the measure $\xi$ on $(\R^d, \B(\R^d))$ by $\xi(A) := \int_A \varrho_0(x) \d x$ for $A\in \B(\R^d)$, such that $\xi(\d x) = \varrho_0 (x) \d x$ and $\ell_{\varrho_0,\varrho_1}(t) = \xi(L(t,\varrho_1))$.
	As $- \ell_{\varrho_0,\varrho_1}$ is easily seen to be right-continuous, the Lebesgue-Stieltjes measure it generates \cite[Section 1.3.2]{Athreya} is determined by mapping, for any $t_1, t_2 \in \ooint{0}{\infty}$ with $t_1 < t_2$, the interval $\ocint{t_1}{t_2}$ to
	\begin{align*}
		&(- \ell_{\varrho_0,\varrho_1})(t_2) - (- \ell_{\varrho_0,\varrho_1})(t_1) \\
		&= \xi(L(t_1,\varrho_1)) - \xi(L(t_2,\varrho_1)) \\
		&= \xi(\{x \in \R^d \mid \varrho_1(x) > t_1\}) - \xi(\{x \in \R^d \mid \varrho_1(x) > t_2\}) \\
		&= \xi(\{x \in \R^d \mid t_1 < \varrho_1(x) \leq t_2\}) \\
		&= (\xi\circ \varrho_1^{-1})(\ocint{t_1}{t_2}) .
	\end{align*}
	Therefore it is given by $\xi\circ \varrho_1^{-1}$, the pushforward measure of $\xi$ w.r.t.~$\varrho_1$.

	For any $B \in \B(\ooint{0}{\infty})$, let us define a function $g_B: \ooint{0}{\infty} \ra \R_+$ by
	\begin{equation*}
		g_B(s)
		:= \frac{\lambda(B \, \cap \ooint{0}{s})}{s}
	\end{equation*}
	and observe that
	\begin{align*}
		g_B(\varrho_1(x))
		&= \frac{\int_B \ind_{\ooint{0}{\varrho_1(x)}}(t)\, \d t}{\varrho_1(x)} \\
		&= \frac{\int_B \ind_{L(t,\varrho_1)}(x)\, \d t}{\varrho_1(x)}
		= U_T(x,B)
	\end{align*}
	for any $x \in \R^d$. Now, by the change of variables formula \cite[Theorem 3.6.1]{Bogachev} and the fact that $\xi\circ \varrho_1^{-1}$ is the Lebesgue-Stieltjes measure generated by $-\ell_{\varrho_0,\varrho_1}$, we get
	\begin{align*}
		P_T(t,B)
		&= \int_{\R^d} U_T(x, B) U_X(t, \d x) \\ 
		&= \frac{\int_{L(t,\varrho_1)} U_T(x,B) \varrho_0(x)  \d x}{\int_{L(t,\varrho_1)} \varrho_0(x) \d x} \\ 
		&= \frac{1}{\ell_{\varrho_0,\varrho_1}(t)} \int_{L(t,\varrho_1)} U_T(x,B) \xi(\d x) \\ 
		&= \frac{1}{\ell_{\varrho_0,\varrho_1}(t)} \int_{\R^d} g_B(\varrho_1(x)) \ind_{\ooint{0}{\varrho_1(x)}}(t) \xi(\d x) \\ 
		&= \frac{1}{\ell_{\varrho_0,\varrho_1}(t)} \int_0^{\infty} g_B(s) \ind_{\ooint{0}{s}}(t)\; (\xi\circ \varrho_1^{-1})(\d s) \\ 
		&= \frac{1}{\ell_{\varrho_0,\varrho_1}(t)} \int_t^{\infty} \frac{\lambda(B \, \cap \ooint{0}{s})}{s} (\xi\circ \varrho_1^{-1})(\d s) \\ 
		&= \frac{1}{\ell_{\varrho_0,\varrho_1}(t)}\int_t^{\infty} \frac{\lambda(B \, \cap \ooint{0}{s})}{s} \d(-\ell_{\varrho_0,\varrho_1})(s) 
	\end{align*}
	for any $t > 0$, $B \in \B(\ooint{0}{\infty})$, which proves the claimed result. 
\end{proof}

By combining the previous two theorems suitably we are able to show that if two distributions have the same function $\ell_{\varrho_0,\varrho_1}$, the spectral gaps of slice sampling for them also coincide.\\[-1ex]

\begin{theorem} \label{Thm:gap_identity_ell}
	For $d, k \in \N$ let $\pi \in \M_1(\R^d)$ and $\nu \in \M_1(\R^k)$ be distributions with not necessarily normalized Lebesgue-densities $\varrho$ and $\eta$ satisfying $\varrho = \varrho_0 \, \varrho_1$ and $\eta = \eta_0 \, \eta_1$ for some measurable functions $\varrho_j \colon \R^d \to \R_+$ and $\eta_j \colon \R^k \to \R_+$ for $j=0,1$.
	If $\ell_{\varrho_0,\varrho_1} \equiv \ell_{\eta_0,\eta_1}$, i.e.,~if $\ell_{\varrho_0,\varrho_1}(t) = \ell_{\eta_0,\eta_1}(t)$ for all $t \in \ooint{0}{\infty}$, then 
	\begin{equation*}
		\gap_{\pi}(P_X^{(\pi)})
		= \gap_{\nu}(P_X^{(\nu)}) ,
	\end{equation*}
	where $P_X^{(\pi)}$ is the transition kernel of slice sampling for $\pi$ based on the factorization $\varrho = \varrho_0 \, \varrho_1$ and $P_X^{(\nu)}$ is the transition kernel of slice sampling for $\nu$ based on the factorization $\eta = \eta_0 \, \eta_1$.
\end{theorem}
\begin{proof}
	By Theorem~\ref{Thm:P_T_identity} and the assumption $\ell_{\varrho_0,\varrho_1} \equiv \ell_{\eta_0,\eta_1}$, we immediately get $P_T^{(\pi)}(t,B) = P_T^{(\nu)}(t,B)$ for all $t \in \ooint{0}{\infty}$ and $B \in \B(\ooint{0}{\infty})$, where $P_T^{(\pi)}$ is the transition kernel of the auxiliary chain $(T_n)_{n \in \N}$ of the slice sampler for $\pi$ and $P_T^{(\nu)}$ the corresponding one for $\nu$. As the kernels of the auxiliary chains coincide, their invariant distributions, say $\widetilde{\pi}, \widetilde{\nu}$ (cf.~Lemma \ref{Lem:aux_chain_inv_dist}), must do so as well, i.e.,~$\widetilde{\pi} \equiv \widetilde{\nu}$.
	Applying Theorem~\ref{Thm:gap_identity_X_T} twice yields
	\begin{equation*}
		\gap_{\pi}(P_X^{(\pi)})
		= \gap_{\widetilde{\pi}}(P_T^{(\pi)})
		= \gap_{\widetilde{\nu}}(P_T^{(\nu)})
		= \gap_{\nu}(P_X^{(\nu)}) .
		\qedhere 
	\end{equation*}
\end{proof}

In contrast to the investigation of \cite{SlavaUSS} the former result shows that two different slice samplers (possibly based on different kinds of factorizations, not just different target distributions) have the same spectral gap as long as their corresponding generalized level-set functions coincide. We illustrate the variability of this result in the following consideration.\\[-1ex]
\begin{example}
	For any $d \in \N$, let $\varrho: \R^d \ra \R_+$ and $\eta: \R \ra \R_+$ be given by
	\begin{equation*}
		\varrho(x) = \norm{x}^{1-d} \exp(-\norm{x}) , 
		\qquad
		\eta(s) 
		= \exp(- c_d \abs{s}) 
	\end{equation*}
	for $x \in \R^d$, $s \in \R$, with $c_d := 2 \sigma_{d-1}(\sph^{d-1})^{-1}$. Let $\pi$ and $\nu$ be the distributions with non-normalized densities $\varrho$ and $\eta$. We now consider PSS for $\pi$ and USS for $\nu$, i.e.,~we factorize $\varrho = p_0 \, p_1$ (cf.~\eqref{Eq:PSS_fac_intro}) and\footnote{By $\mathbf{1}$ we denote the function that returns $1$ regardless of the input argument.} $\eta = \mathbf{1} \cdot \eta$.
	By the polar coordinates formula, see Proposition~\ref{Prop:polar_coords}, we readily obtain
	\begin{align*}
		\ell_{p_0,p_1}(t)
		&= \int_{\R^d} \norm{x}^{1-d} \ind_{\ooint{t}{\infty}}(\exp(-\norm{x})) \d x \\
		&= \sigma_{d-1}(\sph^{d-1}) \cdot \int_0^{\infty} r^{1-d} \ind_{\ooint{-\infty}{-\log t}}(r) r^{d-1} \d r \\
		&= 2 c_d^{-1} \cdot (-\log t)
	\end{align*}
	for $t \in \ooint{0}{1}$ and $\ell_{p_0,p_1} = 0$ for $t \geq 1$.
	Furthermore, for $\eta$ one has
	\begin{align*}
		\ell_{\mathbf{1},\eta}(t)
		&= \int_{\R} \ind_{\ooint{t}{\infty}}(\exp(-c_d \abs{s}) \d s \\
		&= \int_{\R} \ind_{\ooint{-\infty}{c_d^{-1} \cdot (- \log t)}}(\abs{s}) \d s
		= 2 c_d^{-1} \cdot (-\log t) ,
	\end{align*}
	again for $t \in \ooint{0}{1}$, with $\ell_{\mathbf{1},\eta}(t) = 0$ for all $t \geq 1$.
	Overall, this yields
	\begin{equation*}
		\ell_{\varrho_0,\varrho_1}(t)
		= 2 c_d^{-1} \log(t^{-1}) \ind_{\ooint{0}{1}}(t)
		= \ell_{\mathbf{1},\eta}(t)
	\end{equation*}
	for all $t \in \ooint{0}{\infty}$. Hence by Theorem \ref{Thm:gap_identity_ell} the spectral gaps that correspond to the different slice sampling schemes coincide. In particular, from \cite[Example~3.15]{SlavaUSS} we know that $	\gap_{\nu}(P_X^{(\nu)}) \geq 1/2$. 
	Consequently, we obtain for PSS that also $\gap_{\pi}(P_X^{(\pi)}) \geq 1/2$.
\end{example}

The example already indicates how Theorem~\ref{Thm:gap_identity_ell} can be applied to carry the spectral gap from one slice sampling scheme to another. Now, we identify properties of the generalized level set function that allow the formerly stated `carrying over' in a universal fashion, cf.~\cite[Definition~3.9]{SlavaUSS}.\\[-1ex]

\begin{definition} \label{Def:Lambda_k}
	For any $k \in \N$, we define $\Lambda_k$ as the class of continuous functions $\ell \colon \ooint{0}{\infty} \, \ra \R_+$ that satisfy 
	\begin{enumerate}[(i)]
		\item \label{Lambda_1st_item}
		$\lim_{t \ra \infty} \ell(t) = 0$ and $\L := \lim_{t \searrow 0} \ell(t) \in \ocint{0}{\infty}$,
		\item \label{Lambda_2nd_item}
		$\ell$ restricted to its open support
		\begin{equation*}
			\supp(\ell) := \oointv{0}{\sup\{t \in \ooint{0}{\infty} \colon \ell(t) > 0\}}
		\end{equation*}
		is strictly decreasing, and
		\item \label{Lambda_3rd_item}
		the function $g\colon \ooint{0}{\L^{1/k}} \ra \ooint{0}{\infty}$ with $g(r)  =\ell^{-1}(r^k)$ is log-concave.\\[-1ex]
	\end{enumerate}
\end{definition}

\begin{remark}
	Conditions \eqref{Lambda_1st_item} and \eqref{Lambda_2nd_item} together with the assumed continuity of $\ell$ guarantee that $\ell$ restricted to its open support $\supp(\ell)$ maps surjectively onto $I_{\ell} := \ooint{0}{\L}$. As condition \eqref{Lambda_2nd_item} also guarantees injectivity of this restricted function, it must actually be bijective, which gives the existence of the inverse function $\ell^{-1}: I_{\ell} \ra \supp(\ell)$ used in condition \eqref{Lambda_3rd_item}. Observe that, as the inverse of a strictly decreasing function, $\ell^{-1}$ must again be strictly decreasing.
\end{remark}

The properties of the formerly defined classes of functions allow us to construct for $\ell \in \Lambda_k$ a not necessarily normalized density function $\eta \colon \R^k \to \R_+$ for which USS, targeting $\nu\in \M_1(\R^k)$ given by
\begin{equation}
	\nu(A) = \frac{\int_A \eta(z) \d z}{\int_{\R^k} \eta(z)\, \d z} ,
	\qquad A\in \B(\R^k),
	\label{Eq:nu_varphi}
\end{equation}
has a spectral gap of at least $1/(k+1)$ and satisfies $\ell_{\mathbf{1},\eta} \equiv \ell$.
With that and Theorem~\ref{Thm:gap_identity_ell} we can draw conclusions about the spectral gap of generalized slice sampling. The correspondingly formulated statement reads as follows.\\[-1ex]

\begin{theorem} \label{Thm:gap_estimate}
	Given $\varrho_0 \colon \R^d \to \R_+$ and a not necessarily normalized density $\varrho \colon \R^d \to \R_+$, choose $\varrho_1 \colon \R^d \to \R_+$, so that $\varrho = \varrho_0 \,\varrho_1$. Let  $\pi \in \M_1(\R^d)$ be specified by $\varrho$ as in \eqref{Eq:pi}. Let $P_X^{(\pi)}$ be the transition kernel that corresponds to slice sampling for $\pi$ based on $\varrho_0$ and $\varrho_1$. Then, for $k \in \N$ with $\ell_{\varrho_0,\varrho_1} \in \Lambda_k$, we have
	\begin{equation*}
		\gap_{\pi}(P_X^{(\pi)})
		\geq \frac{1}{k+1}.
	\end{equation*}
\end{theorem}
\begin{proof}
	To shorten the notation we set $\ell:=\ell_{\varrho_0,\varrho_1}$ and $\L:= \sup_{t>0} \ell(t) = \lim_{t\searrow 0} \ell(t)$. Fix $k \in \N$ with $\ell \in \Lambda_k$. Let $\kappa := \left( k \, \sigma_{k-1}(\sph^{k-1})^{-1} \L \right)^{1/k} \in \ocint{0}{\infty}$ and define $\phi: \ooint{0}{\kappa} \ra \R$ by
	\begin{equation}
		\phi(r) 
		:= -\log\left( \ell^{-1}\left( \frac{\sigma_{k-1}(\sph^{k-1})}{k} r^k \right) \right).
		\label{Eq:phi}
	\end{equation}
	Then, one readily observes that
	\begin{itemize}
		\item $\phi$ is strictly increasing as composition of the strictly increasing function $r \mapsto \sigma_{k-1}(\sph^{k-1})/k \cdot r^k$ and the strictly decreasing functions $\ell^{-1}$ and $r \mapsto -\log r$,
		\item $\phi$ is convex as composition of the linear function $r \mapsto (\sigma_{k-1}(\sph^{k-1})/k)^{1/k} r$ and the (by Definition~\ref{Def:Lambda_k} \eqref{Lambda_3rd_item}) convex function $r \mapsto -\log \ell^{-1}(r^k)$; and
		\item the inverse of $\phi$ is given by
		\begin{equation}
			\phi^{-1}(s)
			= \left( \frac{k  \cdot \ell(\exp(-s)) }{\sigma_{k-1}(\sph^{k-1})}\right)^{1/k} 
			\label{Eq:phi-1}
		\end{equation}
		for $s > - \log(\sup\{t \in \ooint{0}{\infty} \; \colon \ell(t) > 0 \})$.
	\end{itemize}
	
	Consider $\nu \in \M_1(\R^k)$ as in \eqref{Eq:nu_varphi} determined by the not necessarily normalized Lebesgue-density $\eta \colon \R^k \to \R_+$ given by
	\begin{equation*}
		\eta(x)
		= \exp(-\phi(\norm{x})) \ind_{\ooint{0}{\kappa}}(\norm{x}), \qquad x\in\R^k,
	\end{equation*}
	with $\phi$ from \eqref{Eq:phi}. By the fact that $\phi$ is strictly increasing and convex, \cite[Corollary~3.1]{SlavaUSS} yields that the spectral gap of the transition kernel $P_X^{(\nu)}$ of USS for $\nu$ satisfies
	\begin{equation}
		\gap_{\nu}(P_X^{(\nu)}) 
		\geq \frac{1}{k+1} .
		\label{Eq:basic_gap_est}
	\end{equation}
	
	Now the goal is to show that the level-set function $\ell_{\textbf{1},\eta}$ is identical to $\ell=\ell_{\varrho_0,\varrho_1}$. 
	We obtain for $0 \neq x \in \R^k$ and $t \in \ooint{0}{\sup_{y \in \R^k} \eta(y)}$ that
	\begin{align*}
		\eta(x) > t \quad 
		&\LRA \quad 
		\phi(\norm{x}) < -\log t \quad \text{and} \quad \norm{x} < \kappa \\
		&\LRA \quad 
		\norm{x} < \phi^{-1}(-\log t) \quad \text{and} \quad \norm{x} < \kappa \\
		&\LRA \quad 
		\norm{x} < \phi^{-1}(-\log t) ,
	\end{align*}
	where the second equivalence relies on $\phi^{-1}$ being strictly increasing, and the third equivalence on $\phi^{-1}$ mapping to the domain $\ooint{0}{\kappa}$ of $\phi$, so that in particular $\phi^{-1} < \kappa$. Hence, the super-level set of $\eta$ is
	\begin{equation*}
		L(t,\eta)
		= \{ x\in \R^k \colon \norm{x} < \phi^{-1}(-\log t) \} .
	\end{equation*}
	Consequently, by the polar coordinates formula, see Proposition~\ref{Prop:polar_coords}, we get
	\begin{align*}
		\ell_{\mathbf{1},\eta}(t)
		&= \int_{L(\eta,t)} \mathbf{1}(x) \d x \\
		&= \sigma_{k-1}(\sph^{k-1}) \int_0^{\infty} \ind_{\ccint{0}{\phi^{-1}(-\log t)}}(r)\, r^{k-1}\, \d r \\
		&= \sigma_{k-1}(\sph^{k-1}) \left[\frac{1}{k} r^k \right]_0^{\phi^{-1}(-\log t)} \\
		&= \frac{\sigma_{k-1}(\sph^{k-1})}{k} \phi^{-1}(-\log t)^k = \ell(t),
		\label{Eq:ell_nu_gap_est_thm}
	\end{align*}
	where the last equality follows by plugging in \eqref{Eq:phi-1}.
	
	Finally, by Theorem~\ref{Thm:gap_identity_ell} and \eqref{Eq:basic_gap_est}, we obtain
	\begin{equation*}
		\gap_{\pi}(P_X^{(\pi)})
		= \gap_{\nu}(P_X^{(\nu)})
		\geq \frac{1}{k+1} ,
	\end{equation*}
	which concludes the proof.
\end{proof}

We add an open question about the result. For this the following class of `good', in the sense of the previous theorem, probability measures 
is required.\\[-1ex]
\begin{definition} \label{Def:good_pi}
	Given $\varrho_0\colon \R^d \to \R_+$ and $k\in\N$ define $\Pi_{\varrho_0,k}\subset \M_1(\R^d)$ as a class of probability measures which satisfies that
	\begin{enumerate}
		\item $\pi \in \Pi_{\varrho_0,k}$ is determined by a not necessarily normalized density $\varrho \colon \R^d \to \R_+$ as defined in \eqref{Eq:pi} with $\varrho_1 \colon \R^d \to \R_+$ being chosen so that $\varrho = \varrho_0 \,\varrho_1$; and
		\item $\ell_{\varrho_0,\varrho_1}\in \Lambda_k$.
	\end{enumerate}
\end{definition}
Then, Theorem~\ref{Thm:gap_estimate} yields
\begin{equation*}
	\inf_{\pi\in \Pi_{\varrho_0,k}} \gap_{\pi}(P_X^{(\pi)})
	\geq \frac{1}{k+1} ,
\end{equation*}
i.e., the worst case behavior of the spectral gap on the input class $\Pi_{\varrho_0,k}$ is at least $1/(k+1)$. The question we pose is how good the lower bound actually is. In other words, is there a matching upper bound of the worst case spectral gap that also converges with $k\to \infty$ to zero and if so, does it lead to the fact that the lower bound cannot be improved? Any insight into that direction may lead to a characterization of the spectral gap of generalized slice sampling and therefore indicate its limitations.

We finish this section with an immediate consequence of Theorem~\ref{Thm:gap_estimate} w.r.t.~PSS.\\[-1ex]
\begin{corollary} \label{Cor:PSS}
	For a not necessarily normalized density function $\varrho \colon \R^d \to \R_+$ let $\pi$ be the corresponding distribution as in \eqref{Eq:pi} and define $p_0$, $p_1$ by \eqref{Eq:PSS_fac_intro}. Assume that $\ell_{p_0,p_1} \in \Lambda_k$ for some $k \in \N$. Then $P = P^{(\pi)}_X$, the transition kernel of PSS for $\pi$ satisfies
	\begin{equation*}
		\gap_{\pi}(P)
		\geq \frac{1}{k+1}.
	\end{equation*}
\end{corollary}

\begin{remark}
	In the setting of the previous corollary assume that $k$ does not depend on $d$. In that case we have already a dimension-independent lower bound of the spectral gap of PSS. We want to emphasize that even though the spectral gap is independent of $d$, implementing the 2nd step of Algorithm~\ref{alg} may lead to an acceptance probability that decreases w.r.t.~$d$. The already mentioned Gibbsian polar slice sampler of \cite{GPSS} addresses this issue.\\[-1ex]
\end{remark}

We now move on to our next main result, where we apply Theorem \ref{Thm:gap_estimate} (or Corollary~\ref{Cor:PSS}) and provide concrete properties of $\varrho$ that lead to a spectral gap of at least $1/2$ of PSS.

\subsection{Polar Slice Sampling} \label{Sec:PSS_gap_est}

We assume here $d \geq 2$ and note that PSS coincides with USS for $d=1$. Consequently, in that case, spectral gap estimates for the latter carry over to the former. 

The strategy in this section is to consider a specific class of not necessarily normalized density functions $\varrho$ for which we verify that the corresponding PSS generalized level set function $\ell_{p_0,p_1}$ satisfies $\ell_{p_0,p_1} \in \Lambda_1$. Then, by applying Theorem~\ref{Thm:gap_estimate} we readily obtain a dimension independent lower bound of the spectral gap of PSS.
We formulate the main statement and discuss the required assumptions.\\[-1ex]

\begin{theorem} \label{Thm:PSS_gap_est}
	Let $\pi \in \M_1(\R^d)$ be a distribution with not necessarily normalized density $\varrho$ given by
	\begin{equation*}
		\varrho(x)
		= \exp(- \phi(\norm{x})) \ind_{\ooint{0}{\kappa}}(\norm{x}) ,
	\end{equation*}
	where $\kappa \in \ocint{0}{\infty}$ and $\phi: \ooint{0}{\kappa} \ra \R$ is a convex and twice differentiable function that satisfies $\lim_{r \nearrow \kappa} \phi(r) = \infty$.
	Then we have 
	\begin{equation*}
		\gap_{\pi}(P)
		\geq \frac{1}{2},
	\end{equation*}
	where $P=P^{(\pi)}_X$ is the transition kernel of PSS for $\pi$.\\[-1ex]
\end{theorem}
\begin{remark}
	We discuss the conditions and appearing objects of the theorem:
	\begin{itemize}
		\item The parameter $\kappa$ controls the support of $\varrho$: If it is finite, $\varrho$ is only supported on the zero-centered Euclidean ball of radius $\kappa$, and if it is infinite, $\varrho$ is supported on all of $\R^d$.
		\item The densities $\varrho$ to which the theorem applies are rotationally invariant, i.e.~they may only depend on the function's argument $x$ through $\norm{x}$.
		\item The convexity constraint on $\phi$ gives that $\varrho$  is log-concave along rays emanating from the origin, which already emerged to be a useful property for proving theoretical results regarding PSS in \cite{PolarSS}. In particular, it guarantees that the later appearing function $h_1\colon \ooint{0}{\kappa} \ra \R_+$, $r\mapsto r^{d-1} \exp(-\phi(r))$ has interval-like super level sets.
		\item That the function $\phi$ is required to be twice differentiable eases our proof, but we believe that the theorem's claim is still true without this assumption.
		\item The condition $\lim_{r \nearrow \kappa} \phi(r) = \infty$ means that $\varrho$ must tend to zero whenever its argument approaches the boundary of the support. The requirement is always satisfied when $\kappa=\infty$, since $\varrho$ is assumed to be Lebesgue-integrable. 
	\end{itemize}
\end{remark}	
Note that the rotational invariance, convexity and the constraint $\lim_{r \nearrow \kappa} \phi(r) = \infty$ without any further monotonicity requirements on $\phi$ lead to two types of admissible target densities: Unimodal densities, which result from non-decreasing $\phi$, and ``volcano''-shaped densities, which result from $\phi$ that are initially strictly decreasing and then at some point become strictly increasing.
Particularly notable is that the lower-bound of the spectral gap on the class of $\varrho$ specified in the theorem is constant, i.e., does not depend on any continuity or concentration parameter, not even on the state space dimension $d$.
In the course of proving Theorem~\ref{Thm:gap_estimate} we start with a characterization of the corresponding level set functions.\\[-1ex]

\begin{lemma} \label{Lem:gap_est_PSS_ell_formula}
	Assume that the requirements of Theorem \ref{Thm:PSS_gap_est} are satisfied. Factorize $\varrho$ in accordance with PSS, i.e.,~$\varrho = p_0 \, p_1$ (cf.~\eqref{Eq:PSS_fac_intro}). Define $h_1: \ooint{0}{\kappa} \ra \R_+$ by
	\begin{equation*}
		h_1(r) 
		:= r^{d-1} \exp(-\phi(r)) .
	\end{equation*}
	Then, there exists a value $r_{\text{mode}} \in \ooint{0}{\kappa}$ such that
	\begin{align*}
		&\ell_{p_0,p_1}(t) \\
		&=
		\begin{cases}
			\sigma_{d-1}(\sph^{d-1}) (r_{\max}(t) - r_{\min}(t)) & 0 <t < h_1(r_{\text{mode}}) , \\
			0 & t \geq h_1(r_{\text{mode}}) ,
		\end{cases}
	\end{align*}
	with
	functions
	\begin{align*}
		r_{\max} \colon &\oointv{0}{h_1(r_{\text{mode}})} \ra \oointv{r_{\text{mode}}}{\kappa}, \\
		&\; t \mapsto \left( h_1|_{\ooint{r_{\text{mode}}}{\kappa}} \right)^{-1}(t), \\
		r_{\min} \colon &\oointv{0}{h_1(r_{\text{mode}})} \ra \oointv{0}{r_{\text{mode}}}, \\
		&\; t \mapsto \left( h_1|_{\ooint{0}{r_{\text{mode}}}} \right)^{-1}(t),
	\end{align*}
	that are strictly decreasing and strictly increasing, respectively.
\end{lemma}
\begin{proof}
	The generalized level set function $\ell_{p_0,p_1}$ of $\varrho = p_0 \, p_1$, given as in \eqref{Eq:PSS_fac_intro}, satisfies, by virtue of Proposition \ref{Prop:polar_coords}, for all $t > 0$ that
	\begin{align}
			\ell_{p_0,p_1}(t) 
			&= \int_{\R^d} \norm{x}^{1-d} \ind_{(t,\infty)}(h_1(\norm{x})) \notag \ind_{\ooint{0}{\kappa}}(\norm{x}) \d x \\
			&= \sigma_{d-1}(\sph^{d-1}) \int_0^{\kappa} r^{1-d}\,\ind_{(t,\infty)}(h_1(r)) \, r^{d-1} \d r \notag \\
			&= \sigma_{d-1}(\sph^{d-1}) \lambda(\{r \in \ooint{0}{\kappa} \; \colon h_1(r) > t \}). 
		\label{Eq:ell_incomplete}
	\end{align}
	We analyze the function $h_1$ to deduce the claimed representation of $\ell_{p_0,p_1}$ from the former expression.
	Observe that
	\begin{align*}
		h_1^{\prime}(r)
		&= (d-1) r^{d-2} \exp(-\phi(r)) - r^{d-1} \phi^{\prime}(r) \exp(-\phi(r)) \\
		&= (d - 1 - r \phi^{\prime}(r)) r^{d-2} \exp(-\phi(r)) .
	\end{align*}
	Define $h_2: \ooint{0}{\kappa} \ra \R, \; r \mapsto r \phi^{\prime}(r)$ and observe $h_2^{\prime}(r) = \phi^{\prime}(r) + r \phi^{\prime\prime}(r)$.
	Let $r_{\phi} \in \coint{0}{\kappa}$ be such that $\phi$ is decreasing on $\ooint{0}{r_{\phi}}$ (possibly $\emptyset$) and strictly increasing on $\ooint{r_{\phi}}{\kappa}$. That this value exists is an immediate consequence of the convexity of $\phi$ and the requirement $\lim_{r \nearrow \kappa} \phi(r) = \infty$. Now for $r \in \ooint{0}{r_{\phi}}$, where $\phi$ is decreasing ($\phi^{\prime}(r) \leq 0$), we have $h_2(r) \leq 0$. And for $r \in \ooint{r_{\phi}}{\kappa}$, where $\phi$ is not just convex ($\phi^{\prime\prime}(r) \geq 0$) but also strictly increasing ($\phi^{\prime}(r) > 0$), we get $h_2^{\prime}(r) > 0$. 
	
	In cases where $\kappa < \infty$, the property $\lim_{r \nearrow \kappa} \phi(r) = \infty$ implies $\lim_{r \nearrow \kappa} \phi^{\prime}(r) = \infty$ 
	%
	%
	and thus $\lim_{r \nearrow \kappa} h_2(r) = \infty$. If $\kappa = \infty$, then the same result follows from $\phi^{\prime}(r)$ being positive for $r > r_{\phi}$ and non-decreasing on account of its slope $\phi^{\prime\prime}$ being non-negative by convexity of $\phi$. 
	
	Combining these observations, we see that $h_2$ is upper-bounded by zero on $\ooint{0}{r_{\phi}}$ and strictly increasing towards $+\infty$ on $\ooint{r_{\phi}}{\kappa}$. 
	Therefore, there exists an $r_{\text{mode}} \in \ooint{0}{\kappa}$ such that
	$r\mapsto d - 1 - r \phi^{\prime}(r)$ 
	within $h_1^{\prime}$ is positive on  $\ooint{0}{r_{\text{mode}}}$ with $h_2(r_{\text{mode}}) = d-1$ and negative on  $\ooint{r_{\text{mode}}}{\kappa}$. 
	Consequently, for $r \in \oointv{0}{r_{\text{mode}}}$ we get 
	\begin{equation*}
		h_1^{\prime}(r)
		=  \underbrace{(d - 1 - r \phi^{\prime}(r))}_{>0} \underbrace{r^{d-2}}_{>0} \underbrace{\exp(-\phi(r))}_{>0} > 0,
	\end{equation*}
	whereas for $r \in \oointv{r_{\text{mode}}}{\kappa}$ we have 
	\begin{equation*}
		h_1^{\prime}(r)
		=  \underbrace{(d - 1 - r \phi^{\prime}(r))}_{<0} \underbrace{r^{d-2}}_{>0} \underbrace{\exp(-\phi(r))}_{>0} < 0 .
	\end{equation*}
	In other words $h_1$ is unimodal with mode located at $r_{\text{mode}}$. Moreover, the above shows that $h_1|_{]0,r_{\text{mode}}[}$, the inverse of $r_{\min}$, is a strictly increasing function, which implies that $r_{\min}$ is also strictly increasing. Analogously it follows that $r_{\max}$ is strictly decreasing.
	Besides that we have $h_1(r) < h_1(r_{\text{mode}})$ for $r \neq r_{\text{mode}}$, which by \eqref{Eq:ell_incomplete} readily gives $\ell_{p_0,p_1}(t) = 0$ for $t \geq h_1(r_{\text{mode}})$.
	
	Since $\phi^{\prime}$ is positive on $\ooint{r_{\phi}}{\kappa}$ and non-decreasing (as noted before), there is an $\varepsilon>0$ with $r_\phi+\varepsilon<\kappa$ such that for all $r\geq r_\phi +\varepsilon$ it satisfies $\phi^{\prime}(r) \geq \phi^{\prime}(r_{\phi} + \varepsilon) > 0$. For $\kappa=\infty$, using the former observation, we have	
	by L'Hospital's rule,
	\begin{align*}
		\lim_{r\nearrow\infty} h_1(r)
		&= \lim_{r\nearrow\infty} \frac{r^{d-1}}{\exp(\phi(r))} \\
		&= \lim_{r\nearrow\infty} \frac{(d-1) r^{d-2}}{\phi^{\prime}(r) \exp(\phi(r))} \\ 
		&\leq \frac{d-1}{\phi^{\prime}(r_{\phi} + \varepsilon)} \lim_{r\nearrow\infty} \frac{r^{d-2}}{\exp(\phi(r))}
	\end{align*}
	and by iterating this inductively we get
	\begin{equation*}
		\lim_{r\nearrow\infty} h_1(r)
		\leq \frac{(d-1)!}{\phi^{\prime}(r_{\phi} + \varepsilon)^{d-1}} \lim_{r\nearrow\infty} \frac{1}{\exp(\phi(r))}
		= 0.
	\end{equation*}
	For $\kappa<\infty$ by $\lim_{r \nearrow \kappa} \phi(r) = \infty$ it also  follows $\lim_{r\nearrow\kappa} h_1(r)=0$.
	Consequently we have
	\begin{equation*}
		\lim_{r\searrow 0} h_1(r) 
		= 0
		= \lim_{r\nearrow\kappa} h_1(r),
	\end{equation*}
	where the first equality holds by definition.
	This finally allows us to conclude
	\begin{align*}
		&\{r \in \ooint{0}{\kappa} \; \colon h_1(r) > t \} \\
		&= \oointv{\left( h_1|_{\ooint{0}{r_{\text{mode}}}} \right)^{-1}(t)}{\left( h_1|_{\ooint{r_{\text{mode}}}{\kappa}} \right)^{-1}(t)} \\
		&= \ooint{r_{\min}(t)}{r_{\max}(t)}
	\end{align*}
	for $t \in \oointv{0}{h_1(r_{\text{mode}})}$. The claimed formula for $\ell_{p_0,p_1}$ is obtained by plugging this identity into \eqref{Eq:ell_incomplete}.
\end{proof}

Using the formerly developed tool, we are able to deliver the proof of the theorem.

\begin{proof}[Proof of Theorem \ref{Thm:PSS_gap_est}]
	To verify the statement of Theorem~\ref{Thm:PSS_gap_est} we show that for $\varrho$, satisfying the assumptions formulated there, the corresponding level set function $\ell_{p_0,p_1}$ satisfies $\ell_{p_0,p_1} \in \Lambda_1$. By Lemma~\ref{Lem:gap_est_PSS_ell_formula} it is easily seen that $\ell_{p_0,p_1}$ is a continuous function, such that it is sufficient to check \eqref{Lambda_1st_item}, \eqref{Lambda_2nd_item} and \eqref{Lambda_3rd_item} of Definition~\ref{Def:Lambda_k} for $k=1$.\\
	\textbf{To \eqref{Lambda_1st_item}:} Just by being a generalized level set function, $\ell_{p_0,p_1}$ satisfies the limit properties. \\
	\textbf{To \eqref{Lambda_2nd_item}:}
	The monotonicity properties of $r_{\max}$ and $r_{\min}$ provided by Lemma \ref{Lem:gap_est_PSS_ell_formula} yield that $\ell_{p_0,p_1}$ is strictly decreasing on $\supp(\ell_{p_0,p_1})$. \\
	\textbf{To \eqref{Lambda_3rd_item}:}
	By Proposition \ref{Prop:alternative_cond} it is sufficient to show that
	\begin{equation*}
		h_3(s) := \ell_{p_0,p_1}(\exp(-s))
	\end{equation*} 
	is concave on a set $D$, which, by Lemma \ref{Lem:gap_est_PSS_ell_formula}, is here given by $D = \ooint{-\log h_1(r_{\text{mode}})}{\infty}$. Using the lemma's representation of $\ell_{p_0,p_1}$, we can rewrite $h_3$ as
	\begin{equation*}
		h_3(s) 
		= \sigma_{d-1}(\sph^{d-1}) (r_{\max}(\exp(-s)) - r_{\min}(\exp(-s))).
	\end{equation*}
	Consequently, the concavity of $h_3$ follows by concavity of 
	\begin{align*}
		h_4: 
		&\oointv{-\log h_1(r_{\text{mode}})}{\infty} \ra \oointv{r_{\text{mode}}}{\kappa}, \\
		&\; s \mapsto r_{\max}(\exp(-s))
	\end{align*}
	as well as convexity of
	\begin{align*}
		h_5:
		&\oointv{-\log h_1(r_{\text{mode}})}{\infty} \ra \oointv{0}{r_{\text{mode}}}, \\
		&\; s \mapsto r_{\min}(\exp(-s)) .
	\end{align*}
	Concavity of $h_4$ means convexity of $-h_4$. Clearly, $-h_4$ is continuous and as $h_4$ is the composition of two strictly decreasing functions, it is strictly increasing, so $-h_4$ is strictly decreasing. By Lemma \ref{Lem:convex_inverse}, convexity of $-h_4$ is equivalent to convexity of its inverse $r \mapsto h_4^{-1}(-r)$, which in turn is equivalent to convexity of $h_4^{-1}$ itself (as the graph of one of these functions is just a reflection of that of the other on the axis $r = 0$), which is given by
	\begin{align*}
		h_4^{-1}(r)
		&= - \log r_{\max}^{-1}(r) 
		= - \log h_1|_{\ooint{r_{\text{mode}}}{\kappa}}(r) \\
		&= - \log( r^{d-1} \exp(-\phi(r)) ) 
		= \phi(r) - (d-1) \log r .
	\end{align*}
	However, since $\phi$ is convex by assumption and $-\log$ is known to be convex, the convexity of $h_4^{-1}$ is obvious. 
	
	As the composition of a strictly increasing and a strictly decreasing function, $h_5$ is strictly decreasing. Because $h_5$ is clearly also continuous, applying Lemma \ref{Lem:convex_inverse} again yields that the convexity of $h_5$ is equivalent to that of its inverse $h_5^{-1}$, which is given by
	\begin{align*}
		h_5^{-1}(r)
		&= - \log r_{\min}^{-1}(r) 
		= - \log h_1|_{]0, r_{\text{mode}}[}(r) \\
		&= - \log( r^{d-1} \exp(-\phi(r)) ) 
		= \phi(r) - (d-1) \log r .
	\end{align*}
	Thus, the convexity of $h_5^{-1}$ follows by the same argument as that of $h_4^{-1}$.
	
	Therefore \eqref{Lambda_3rd_item} for $k=1$ is proven and $\ell_{p_0,p_1}\in \Lambda_1$. By Theorem \ref{Thm:gap_estimate} this implies the claimed spectral gap estimate.
\end{proof}

\section{Concluding Remarks} \label{Sec:Con}

Driven by empirically observed dimension independent IAT behavior, as documented in the motivating illustration in Section~\ref{Sec:Intro}, and the recent algorithmic contribution about Gibbsian polar slice sampling, we investigated the spectral gap of PSS. For arbitrary dimension, if $\varrho$, the possibly not normalized density function of the distribution of interest, is rotationally invariant, log-concave along rays emanating from the origin and sufficiently smooth we proved a lower bound of $1/2$ on the spectral gap.
Along the way we significantly extended the theory of \cite{SlavaUSS} into the setting of general slice sampling that is based on a factorization $\varrho = \varrho_0 \, \varrho_1$. In Definition~\ref{Def:Lambda_k} we presented a class of functions $\Lambda_k$, already introduced in \cite[Definition~3.9]{SlavaUSS}, which 
provides the required conditions on the level set function $\ell_{\varrho_0,\varrho_1}$ for verifying the lower bound $1/(k+1)$ of the spectral gap for generalized slice sampling. As an immediate consequence this lower bound can be applied in the PSS-setting. Moreover, it served as the main tool for proving the aforementioned dimension-independent spectral gap estimate for PSS.

We point to open questions, limitations and some directions of future work. Let us start with the question that has already been formulated after Definition \ref{Def:good_pi} on how `good' the lower bound of the spectral gap of generalized slice sampling of Theorem~\ref{Thm:gap_estimate} actually is. We conjecture that at least for some $\varrho_0$ on the class $\Pi_{\varrho_0,k}$ the result cannot be qualitatively improved. We surmise that there is an upper bound `function' $u\colon \N \to \R_+$ with $\lim_{k \to \infty} u(k) = 0$ such that the worst case spectral gap satisfies
\begin{equation*}
	\frac{1}{k+1} \leq	\inf_{\pi \in \Pi_{\varrho_0,k}} \gap_{\pi}(P^{(\pi)}_X) \leq u(k) . 
\end{equation*}
By proving this conjecture, one would show that the parameter $k$ is the right quantity for characterizing the spectral gap of general slice sampling, which points to the limitation that for large $k$ the `efficiency' of slice sampling indeed deteriorates.
Related to the understanding of the limitations of Theorem~\ref{Thm:gap_estimate} one may ask whether an extension into a manifold setting is possible. Recently there have been investigations of slice sampling approaches on the sphere, see e.g.~\cite{SphericalSS,lie2021dimension}, which may serve as a starting point into that direction.

Regarding our explicit dimension-independent spectral gap estimate for PSS, it is reasonable to ask how the proven estimate generalizes to broader classes of target densities, for example rotationally asymmetric ones, or those not centered around the origin. Neither rotational invariance nor being centered around the origin are properties that are exploited in the generic algorithmic description of PSS. Therefore, those seem to be merely exploited within our analysis technique rather than necessary. It would be interesting to find other, more commonly used properties, e.g., strong concavity of smooth log-densities, that yield dimension independent convergence results. Unfortunately, the proof of our result cannot readily be adapted to such cases and it is unknown how the spectral gap behaves.

As explained before, a crucial motivation for studying PSS is Gibbsian polar slice sampling, introduced in \cite{GPSS}. It can be considered a hybrid slice sampler, cf.~\cite{Latuszynski}, that mimics PSS. Under suitable assumptions, it has been shown in \cite{Latuszynski} that hybrid uniform slice sampling has a positive spectral gap whenever USS has one. Explicit lower bounds of the gap are to some extent inherited. It is of course very natural to ask for an extension of this result regarding PSS and the Gibbsian approach.

\section*{Acknowledgements}

We thank the anonymous reviewers for their helpful remarks. PS gratefully acknowledges funding by the Carl Zeiss Stiftung within the program ``CZS Stiftungsprofessuren'' and the project ``Interactive Inference''. Moreover, we are grateful for the support of the DFG within project 432680300 -- Collaborative Research Center 1456 ``Mathematics of Experiment'', subproject B02.

\appendix
\section{Auxiliary Results}

\begin{lemma} \label{Lem:update_kernels_adjoint}
	In the setting of Section \ref{Sec:gap_est_tool}, interpret $U_T$ and $U_X$ as linear operators mapping $L_2(\widetilde{\pi}) \ra L_2(\pi)$ and $L_2(\pi) \ra L_2(\widetilde{\pi})$, then $U_X$ is the adjoint operator of $U_T$, meaning 
	\begin{equation*}
		\langle U_T g, h \rangle_{\pi} = \langle g, U_X h \rangle_{\widetilde{\pi}}
	\end{equation*}
	for all $g \in L_2(\widetilde{\pi})$, $h \in L_2(\pi)$.
\end{lemma}
\begin{proof}
	Fix $g \in L_2(\widetilde{\pi})$ and $h \in L_2(\pi)$, then
	\begin{align*}
		&\langle U_T g, h \rangle_{\pi} \\
		&= \int_{\R^d} (U_T g)(x) h(x) \pi(\d x) \\ 
		&= \int_{\R^d} \int_{\ooint{0}{\infty}} g(t) \frac{\ind_{L(t,\varrho_1)}(x) \d t}{\varrho_1(x)} \; h(x) \, C^{-1} \varrho_0(x) \varrho_1(x) \d x \\ 
		&= \int_{\ooint{0}{\infty}} g(t) \, C^{-1} \int_{L(t,\varrho_1)} h(x) \varrho_0(x) \d x \d t \\ 
		&= \int_{\ooint{0}{\infty}} g(t) (U_X h)(t) \cdot \widetilde{\varrho}(t) \d t 
		= \langle g, U_X h \rangle_{\widetilde{\pi}} . 
		\qedhere 
	\end{align*}
\end{proof}

\begin{lemma} \label{Lem:linop_identities}
	In the setting of Section \ref{Sec:gap_est_tool}, interpret the measures $\pi$, $\widetilde{\pi}$ as kernels and view these as linear operators (cf.~\eqref{Eq:pi_as_operator}). Then we have the following relations between the operators\footnote{In \eqref{it: aux_2} and \eqref{it: aux_3} we consider $\pi$ either as an operator mapping $L_2(\pi) \ra L_2(\pi)$ or as one mapping $L_2(\pi) \ra L_2(\widetilde{\pi})$, which we may do because $\pi$ only maps to constant functions. The same applies for $\widetilde{\pi}$ by interchanging the roles of $L_2(\pi)$ and $L_2(\widetilde{\pi})$.} $\pi$, $\widetilde{\pi}$, $U_T$ and $U_X$:
	\begin{enumerate}[(i)]
		\item \label{it: aux_1}
		$\pi$ is the adjoint operator of $\widetilde{\pi}$,
		\item \label{it: aux_2}
		$U_T \pi = \widetilde{\pi} U_X = \widetilde{\pi} \pi = \pi$,
		\item \label{it: aux_3}
		$U_X \widetilde{\pi} = \pi U_T = \pi \widetilde{\pi} = \widetilde{\pi}$.
	\end{enumerate}
\end{lemma}
\begin{proof}
	We frequently emphasize that $\pi$ and $\widetilde{\pi}$ only map to constant functions by dropping the argument in our notation.
	To prove the claims, we arbitrarily fix $g \in L_2(\widetilde{\pi})$, $h \in L_2(\pi)$, $x \in \R^d$ and $t \in \ooint{0}{\infty}$. \\
	\textbf{To \eqref{it: aux_1}:} We have
	\begin{align*}
		\langle \widetilde{\pi} g, h \rangle_{\pi}
		&= \int_{\R^d} (\widetilde{\pi} g)(y) h(y) \pi(\d y) 
		= (\widetilde{\pi} g) \cdot (\pi h) \\ 
		&= \int_{\ooint{0}{\infty}} g(s) \, (\pi h)(s) \, \widetilde{\pi}(\d s) 
		= \langle g, \pi h \rangle_{\widetilde{\pi}} . 
	\end{align*}
	\textbf{To \eqref{it: aux_2}:} Observe 
	\begin{align*}
		(U_T \pi h)(x)
		&= \int_{\ooint{0}{\infty}} (\pi h)(s) U_T(x, \d s) \\ 
		&= (\pi h) \int_{\ooint{0}{\infty}} U_T(x, \d s) 
		= \pi h , 
	\end{align*}
	showing $U_T \pi = \pi$, and
	\begin{align*}
		\widetilde{\pi} U_X h
		&= \int_{\ooint{0}{\infty}} (U_X h)(s) \widetilde{\pi}(\d s) \\ 
		&= \int_{\ooint{0}{\infty}} \frac{\int_{L(s,\varrho_1)} h(y) \varrho_0(y) \d y}{\int_{L(s,\varrho_1)} \varrho_0(y) \d y} \widetilde{\pi}(\d s) \\ 
		&= C^{-1} \int_{\ooint{0}{\infty}} \int_{L(s,\varrho_1)} h(y) \varrho_0(y) \d y \d s \\ 
		&= C^{-1} \int_{\R^d} h(y) \varrho_0(y) \int_{\ooint{0}{\infty}} \ind_{L(s,\varrho_1)}(y) \d s \d y \\ 
		&= C^{-1} \int_{\R^d} h(y) \varrho_0(y) \varrho_1(y) \d y \\ 
		&= \int_{\R^d} h(y) \pi(\d y) 
		= \pi h , 
	\end{align*}
	showing $\widetilde{\pi} U_X = \pi$. Finally
	\begin{equation*}
		\widetilde{\pi} \pi h 
		= \int_{\ooint{0}{\infty}} (\pi h)(s) \widetilde{\pi}(\d s) 
		= (\pi h) \; \int_{\ooint{0}{\infty}} \widetilde{\pi}(\d s) 
		= \pi h ,
	\end{equation*}
	shows $\widetilde{\pi} \pi = \pi$, such that the claimed equalities readily follow. \\
	\textbf{To \eqref{it: aux_3}:} Similarly as before we have
	\begin{align*}
		(U_X \widetilde{\pi} g)(t)
		&= \int_{\R^d} (\widetilde{\pi} g)(y) U_X(t, \d y) \\
		&= (\widetilde{\pi} g) \int_{\R^d} U_X(t, \d y)
		= \widetilde{\pi} g ,
	\end{align*}
	showing $U_X \widetilde{\pi} = \widetilde{\pi}$, and
	\begin{align*}
		\pi U_T g
		&= \int_{\R^d} (U_T g)(y) \pi(\d y) \\ 
		&= C^{-1} \int_{\R^d} \int_{\ooint{0}{\infty}} g(s) \frac{\ind_{L(s,\varrho_1)}(y) \d s}{\varrho_1(y)} \varrho_0(y) \varrho_1(y) \d y \\ 
		&= C^{-1} \int_{\R^d} \int_{\ooint{0}{\infty}} g(s) \varrho_0(y) \ind_{L(s,\varrho_1)}(y) \d s \d y \\ 
		&= \int_{\ooint{0}{\infty}} g(s) \widetilde{\varrho}(s) \d s 
		= \widetilde{\pi} g ,
	\end{align*}
	showing $\pi U_T = \widetilde{\pi}$. Finally 
	\begin{equation*}
		\pi \widetilde{\pi} g
		= \int_{\R^d} (\widetilde{\pi} g)(x) \pi(\d x)
		= (\widetilde{\pi} g) \int_{\R^d} \pi(\d x)
		= \widetilde{\pi} g ,
	\end{equation*}
	shows $\pi \widetilde{\pi} = \widetilde{\pi}$, such that the claimed equalities readily follow.
	\qedhere
\end{proof}

\begin{proposition} \label{Prop:polar_coords}
	For any measurable and integrable function $g: (\R^d,\B(\R^d)) \ra (\R,\B(\R))$, we have
	\begin{equation*}
		\int_{\R^d} g(x) \d x
		= \int_{\sph^{d-1}} \int_0^{\infty} g(r \theta) r^{d-1} \d r \sigma_{d-1}(\d \theta) .
	\end{equation*}
\end{proposition}
\begin{proof}
	See \cite[Theorem 15.13]{Schilling}.
\end{proof}

\begin{lemma} \label{Lem:convex_inverse}
	Assume $h: I_1 \ra I_2$ is a strictly monotone, continuous function mapping between open intervals $I_1, I_2 \subseteq \R$. Then it is clear that $h$ maps bijectively onto its image, w.l.o.g.~$I_2$, with the inverse $h^{-1}: I_2 \ra I_1$ having the same monotonicity property. Moreover,
	\begin{itemize}
		\item if $h$ is increasing, then it is convex if and only if $h^{-1}$ is concave
		\item if $h$ is decreasing, then it is convex if and only if $h^{-1}$ is convex.
	\end{itemize}
\end{lemma}
\begin{proof}
	Denote by $\alpha \in \ccintv{0}{1}$ and $r_1, r_2 \in I_1$ arbitrary elements of the respective sets, then, due to the bijectivity of $h$ onto $I_2$, the values $s_i := h(r_i)$, $i=1,2$, are arbitrary elements of $I_2$. Now
	\begin{align*}
		&h \text{ convex} \\
		&\LRA \quad h(\alpha r_1 + (1-\alpha) r_2) \leq \alpha h(r_1) + (1-\alpha) h(r_2) \\ 
		&\LRA \quad h(\alpha h^{-1}(s_1) + (1-\alpha) h^{-1}(s_2)) \leq \alpha s_1 + (1-\alpha) s_2 \\ 
		&\LRA \quad \alpha h^{-1}(s_1) + (1-\alpha) h^{-1}(s_2) \\
		&\qquad\qquad\quad 
		\begin{cases} 
			\leq h^{-1}(\alpha s_1 + (1-\alpha) s_2) & h^{-1} \text{ increasing} \\
			\geq h^{-1}(\alpha s_1 + (1-\alpha) s_2) & h^{-1} \text{ decreasing} 
		\end{cases} \\ 
		&\LRA \quad 
		\begin{cases}
			h^{-1} \text{ concave} & h \text{ increasing} \\
			h^{-1} \text{ convex} & h \text{ decreasing},
		\end{cases} 
	\end{align*}
	where in the last step we use the fact that $h$ and $h^{-1}$ have the same monotonicity property.
\end{proof}

\begin{proposition} \label{Prop:alternative_cond}
	Let $\ell: \ooint{0}{\infty} \ra \R_+$ be a continuous function that satisfies conditions \eqref{Lambda_1st_item} and \eqref{Lambda_2nd_item} of Definition \ref{Def:Lambda_k} and let $k \in \N$. Then the following two statements are equivalent:
	\begin{enumerate}
		\item The function $s \mapsto \ell(\exp(-s))^{1/k}$ is concave on the domain 
		\begin{equation*}
			D := \oointv{- \log \sup\{t \in \ooint{0}{\infty} \; \colon \ell(t) > 0\}}{\infty} .
		\end{equation*}
		\item The function $\ell$ satisfies condition \eqref{Lambda_3rd_item} of Definition \ref{Def:Lambda_k}.
	\end{enumerate}	
\end{proposition}
\begin{proof}
	By definition of log-concavity, we may rephrase condition \eqref{Lambda_3rd_item} as convexity of the function $h: \ooint{0}{\L^{1/k}} \; \ra D$ defined by
	\begin{equation*}
		h(r) 
		:= - \log \ell^{-1}(r^k) .
	\end{equation*}
	It is easy to verify that the inverse $h^{-1}: D \ra \ooint{0}{\L^{1/k}}$ of $h$ is given by
	\begin{equation*}
		h^{-1}(s)
		= \ell(\exp(-s))^{1/k} .
	\end{equation*}
	Because $h$ is composed of continuous functions, $h$ is also continuous. Moreover, since $h$ is a composition of the strictly decreasing functions $- \log$ and $\ell^{-1}$ and the strictly increasing function $r \mapsto r^k$, it is strictly increasing. Hence, by Lemma \ref{Lem:convex_inverse}, convexity of $h$ is equivalent to concavity of $h^{-1}$, which is precisely the condition stated in the proposition.
\end{proof}


\end{document}